[1994/12/01]
\documentclass[10pt, reqno]{amsart}
\usepackage[english, activeacute]{babel}
\usepackage{amsmath,amsthm,amsxtra}
\usepackage{epsfig}
\usepackage{amssymb}
\usepackage{latexsym}
\usepackage{amsfonts}
\usepackage{hyperref}
\usepackage{pdftricks}

\title{The Gardner equation and the $L^2$-stability of the $N$-soliton solution of the Korteweg-de Vries equation}
\author{Miguel A. Alejo}
\author{Claudio Mu\~noz}
\author{Luis Vega}
\address{Departamento de Matem\'aticas, Facultad de Ciencia y Tecnolog\'ia, Universidad del Pa\'is Vasco, Bilbao \\  Espa\~na}
\email{miguelangel.alejo@ehu.es}
\email{Claudio.Munoz@math.uvsq.fr}
\email{luis.vega@ehu.es}
\date{January, 2011}
\thanks{}

\chardef\bslash=`\\ 

\newtheorem{thm}{Theorem}[section]
\newtheorem{cor}[thm]{Corollary}

\newtheorem{prop}[thm]{Proposition}

\theoremstyle{definition}

\theoremstyle{remark}

\numberwithin{equation}{section}

\newcommand{\R}{\mathbb{R}}

\newcommand{\T}{\mathbb{T}}

\newcommand{\al}{\alpha}
\newcommand{\ga}{\gamma}

\def\bm{\left( \begin{array}{cc}}
\def\endm{\end{array}\right)}

\newcommand{\be}{\begin{equation}}
\newcommand{\ee}{\end{equation}}
\newcommand{\ba}{\begin{equation*}}
\newcommand{\ea}{\begin{equation*}}
\newcommand{\bea}{\begin{eqnarray}}
\newcommand{\eea}{\end{eqnarray}}
\newcommand{\bee}{\begin{eqnarray*}}
\newcommand{\eee}{\end{eqnarray*}}
\newcommand{\ben}{\begin{enumerate}}
\newcommand{\een}{\end{enumerate}}
\newcommand{\nonu}{\nonumber}

\newcommand{\eval}[2][\right]{\relax
  \ifx#1\right\relax \left.\fi#2#1\rvert}

\begin{document}
\begin{abstract}
Multi-soliton solutions of the Korteweg-de Vries equation (KdV) are shown to be globally $L^2$-stable, and asymptotically stable in the sense of Martel-Merle \cite{MMnon}. The proof is surprisingly simple and combines the Gardner transform, which links the Gardner and KdV equations, together with the Martel-Merle-Tsai and Martel-Merle recent results on stability and asymptotic stability in the energy space \cite{MMT,MMan}, applied this time to the Gardner equation. As a by-product, the results of Maddocks-Sachs \cite{MS}, and Merle-Vega \cite{MV} are improved in several directions.  
\end{abstract}
\maketitle \markboth{$L^2$-stability of multi-solitons} {Miguel A. Alejo, Claudio Mu\~noz and Luis Vega}
\renewcommand{\sectionmark}[1]{}

\section{Introduction and Main results}

\medskip

In this paper we consider the nonlinear $L^2$-stability, and asymptotic stability, of the $N$-\emph{soliton} of the Korteweg-de Vries (KdV) equation
\be\label{KdV}
u_{t}  +  (u_{xx} + u^2)_x =0.
\ee
Here $u=u(t,x)$ is a real valued function, and $(t,x)\in \R^2$. This equation arises in Physics as a model of propagation of dispersive long waves, as was pointed out by Russel in 1834 \cite{Miura}. The exact formulation of the KdV equation comes from Korteweg and de Vries (1895) \cite{KdV}. This equation was studied in a numerical work by Fermi, Pasta and Ulam, and by Kruskal and Zabusky \cite{FPU,KZ}.

\medskip

From the mathematical point of view, equation (\ref{KdV}) is an \emph{integrable model} \cite{AKNS,AS,LAX1}, with infinitely many conservation laws. Moreover, since the Cauchy problem associated to (\ref{KdV}) is locally well posed in $L^2(\R)$ (cf. \cite{B1}), each solution is indeed global in time thanks to the \emph{Mass} conservation
\be\label{M}
M[u](t) := \frac 12 \int_\R u^2(t,x)dx =  M[u](0).
\ee
Another important conserved quantity, defined for $H^1(\R)$-valued solutions, is given by the \emph{Energy}
\be\label{E}
E[u](t) := \frac 12 \int_\R u_x^2(t,x) dx -\frac 13 \int_\R u^3(t,x) dx = E[u](0).
\ee
On the other hand, equation (\ref{KdV}) has solitary wave solutions called \emph{solitons}, namely solutions of the form
\be\label{Sol}
u(t,x) = Q_c (x-ct), \quad Q_c(s) := c Q(\sqrt{c} s), \quad c>0,
\ee
and
\be\label{Q}
Q (s):= \frac{3}{1+\cosh(s)} .
\ee

\medskip

The study of perturbations of solitons or solitary waves lead to the introduction of the concepts of \emph{orbital and asymptotic stability}. In particular, since energy and mass are conserved quantities, it is natural to expect that solitons are stable in the energy space $H^1(\R)$. Indeed, $H^1$-stability of KdV solitons has been considered in \cite{Benj, BSS}. On the other hand, the asymptotic stability has been studied e.g. in \cite{PW,MMnon}. 

\medskip

Concerning the more involved case of the sum of $N (\geq 2)$ decoupled solitons, stability and asymptotic stability results are very recent.  First of all, let us recall that, as a consequence of the integrability property, KdV allows the existence of solutions behaving, as time goes to infinity, as the sum of $N$ decoupled solitons. These solutions are well-known in the literature and are called \emph{$N$-solitons}, or generically \emph{multi-solitons} \cite{HIROTA}. Indeed, any $N$-soliton solution has the form $u(t,x) := U^{(N)}(x ; c_j, x_j - c_j t),$ where
\be\label{NS}
\big\{U^{(N)}(x; c_j, y_j) \, : \, c_j>0 , \, y_j\in \R, \, j=1,\ldots, N \big\}
\ee
is the family of explicit $N$-soliton profiles (see e.g. Maddocks-Sachs \cite{MS}, \S 3.1). In particular, this solution describes multiple soliton's collisions, but since solitons for KdV equation interact in a linear fashion, there is no residual appearing after the collisions, even if the equation is nonlinear in nature. This is also a consequence of the integrability property. 

\medskip

In \cite{MS}, the authors considered the $H^N(\R)$-stability of the $N$-soliton solution of KdV, by using $N$-conservation laws. Their approach strongly invokes the integrability of the KdV equation, and therefore, in order to enlarge the class of perturbations allowed, a more general method was needed. Precisely, in \cite{MMT, MMan}, the authors improved the preceding result by proving stability and asymptotic stability of the \emph{sum of $N$ solitons}, \emph{well decoupled} at the initial time, in the energy space. Their proof also applies for general nonlinearities and not only for the integrable cases, provided they have stable solitons, in the sense of Weinstein \cite{We}. Note that the \emph{well-preparedness} restriction on the initial data is by now necessary since there is no satisfactory collision theory for the non-integrable cases.\footnote{It turns out that Martel, Merle and the second author of this paper have succeed to describe the collision of two solitons for gKdV equations in some asymptotic regimes and with general nonlinearities beyond the integrable cases, see e.g. \cite{MMcol1,MMcol2,MMcol3,Mu1}.} The Martel-Merle-Tsai approach is based on the construction of $N$ \emph{almost conserved quantities}, related to the mass of each solitary wave, plus the total energy of the solution. Further developments on the $H^1$-stability theory can be found e.g. in  \cite{Bona}.

\medskip

As far as we know, the unique stability result for KdV solitons, below $H^1(\R)$, was proved by Merle and Vega in \cite{MV}. Precisely, in this work, the authors prove that solitons of (\ref{KdV}) are $L^2$-stable, by using the \emph{Miura transform} 
\be\label{Mi}
M[v] := \frac 3{\sqrt{2}} v_x - \frac 32 v^2,
\ee
which links solutions of the \emph{defocusing, modified KdV} equation,
\be\label{dmKdV}
v_t + (v_{xx} -v^3)_x=0, \quad v=v(t,x)\in \R, \quad (t,x)\in \R^2,
\ee
with solutions of the KdV equation (\ref{KdV}). In particular, the image of the family of \emph{kink} solutions of (\ref{dmKdV}) under the transformation (\ref{Mi}) is the soliton $Q_c$ above described, modulo a standard Galilean transformation (cf. \cite{MV}). Since the kink solution of (\ref{dmKdV}) is  $H^1$-stable (see e.g. \cite{Z, MV}), after a local inversion argument, the authors concluded the $L^2$-stability of the KdV soliton. Other applications of the Miura transform are local well and ill-posedness results (cf. \cite{KPV2,CKSTT}). However, the stability property in the case of $H^s$-perturbations, $s\neq 0,1$ is by now a very difficult and open problem. 

\medskip

The Merle-Vega's idea has been applied to different models describing several phenomena. A similar Miura transform is available for the KP II equation, a two-dimensional generalization of the KdV equation. In this case, the transform has an additional term which takes into account the second variable $y$. This property has been studied by Wickerhauser in \cite{W}, and used by Kenig and Martel in \cite{KM} in order to obtain well-posedness results. Finally, Mizumachi and Tzvetkov have shown the stability of solitary waves of KdV, seen as solutions of KP II, under periodic transversal perturbations \cite{MT} (see also Section \ref{5} for some additional remarks on this subject). For instability results, see e.g. \cite{RT}. Finally, we recall the $L^2$-stability result for solitary waves of the cubic NLS proved by Mizumachi and Pelinovsky in \cite{MP}. Now the proof introduces a \emph{B\"acklund transform} linking the zero and the solitary wave solutions.

\medskip

A natural question to consider is the generalization of the Merle-Vega's result to the case of multi-soliton solutions. In \cite{Dirac} (see also \cite{GS}), the author states that the Miura transform sends \emph{multi-kink} solutions of (\ref{dmKdV}) towards a well defined family of \emph{multi-soliton} solutions of (\ref{KdV}). However, we have found that multi-kinks are hard to manipulate, due to the continuous interaction of non-local terms (recall that a kink does not belong to $L^2(\R)$). Therefore we will follow a different approach.

\medskip

Indeed, in this work we invoke a \emph{Gardner transform} \cite{Ga0,Ga}, well-known in the mathematical and physical literature since the late sixties, and which links $H^1$-solutions of the Gardner equation\footnote{In this part we follow the notation of \cite{Mu1}.}
\be\label{Ga}
v_t + (v_{xx} + v^2 -\beta v^3 )_x=0, \; \hbox{ in } \;  \R_t \times \R_x , \quad \beta>0,
\ee
with $L^2$-solutions of the KdV equation (\ref{KdV}). The explicit formula of this transform is given in (\ref{Miura}). Let us recall that the Gardner equation is also an integrable model \cite{Ga}, with soliton solutions of the form 
$$
v(t,x) := Q_{c,\beta} (x-ct),
$$
and\footnote{See e.g. \cite{CGD,Mu1} and references therein for a more detailed description of solitons and integrability for the Gardner equation.} 
\be\label{SolG}
Q_{c,\beta} (s) := \frac{3c}{1+ \rho \cosh(\sqrt{c}s)}, \quad \hbox{ with }\quad  \rho := (1-\frac 92 \beta c)^{1/2}, \quad 0<c < \frac{2}{9\beta}.
\ee
In particular, in the formal limit $\beta\to 0$, we recover the standard KdV soliton (\ref{Sol})-(\ref{Q}). On the other hand, the Cauchy problem associated to (\ref{Ga}) is globally well-posed under initial data in the energy class $H^1(\R)$ (cf. \cite{KPV}), thanks to the mass (\ref{M}) and \emph{energy} conservation laws.

\medskip

We are interested in the image of the family of solutions (\ref{SolG}) under the aforementioned, Gardner transform. Surprisingly enough, it turns out that the resulting family is {\bf nothing but} the KdV soliton family (\ref{Sol}), see (\ref{QcbQc}) below. This formally suggests that multi-soliton solutions of the Gardner equation (\ref{Ga}) are sent towards (or close enough to) multi-soliton solutions of the KdV model (\ref{KdV}), as is done  in \cite{Dirac} for the case of the Miura transform.  

\medskip

In this paper, we profit of this property to improve the $H^1$-stability and asymptotic stability properties proved by Martel, Merle and Tsai in \cite{MMT}, and Martel and Merle \cite{MMan}, now in the case of $L^2$-perturbations of the KdV multi-solitons. We first start with the case of an initial datum close enough to the {\bf sum of $N$ decoupled solitons} of the KdV equation. Our result is the following

\medskip

\begin{thm}[$L^2$-stability of the sum of $N$ solitons of KdV]\label{T1}~

Let $N\geq 2$ and $0<c_1^0<c_2^0<\ldots <c_N^0$. There exist parameters $\al_0, A_0, L, \ga>0$, such that the following holds.  Consider $u_0 \in L^2(\R)$, and assume that there exist $L>L_0$, $\al \in (0,\al_0)$ and $x_1^0<x_2^0<\ldots <x_N^0$, such that 
$$
 x_{j}^0>x_{j-1}^0 + L,\quad  \hbox{with} \quad  j=2, \ldots, N,
$$ 
and
\be\label{In}
\|  u_0 - R_0 \|_{L^2(\R)} \leq \al, \quad \hbox{ with } \quad R_0:= \sum_{j=1}^N Q_{c_j^0} (\cdot -x_j^0).
\ee
Then there exist $x_1(t), \ldots x_N(t)$ such that the solution $u(t)$ of the Cauchy problem for the KdV equation (\ref{KdV}), with initial data $u_0$, satisfies
\ben
\item \emph{Stability.}
\be\label{Fn}
\sup_{t\geq 0}\big\| u(t) - \sum_{j=1}^N Q_{c_j^0}(\cdot - x_j(t)) \big\|_{L^2(\R)}\leq A_0 (\al + e^{-\ga_0 L}).
\ee
\item \emph{Asymptotic stability.} 

\noindent
There exist $c_j(t)>0 $ and possibly a new set of $x_j(t)\in \R$, $j=1,\ldots, N$, such that 
\be\label{ASn}
\lim_{t\to +\infty}\big\| u(t) - \sum_{j=1}^N Q_{ c_j(t)}(\cdot - x_j(t)) \big\|_{L^2(x\geq \frac {c_1^0}{10} t)}=0.
\ee
Moreover, for all  $j=1,\ldots, N$ one has that $ \lim_{t\to +\infty } c_j(t)=: c_j^+ >0$ exists and satisfies 
$$
\sum_{j=1}^N | c_j^+ - c_j^0|\leq KA_0 (\al + e^{-\ga_0 L}),
$$
for some constant $K>0$.
\een
\end{thm}

\medskip

Before explaining the main ideas behind the proof of this result, some remarks are in order.
\medskip

\noindent
{\bf Remarks.}

\noindent
1. Compared with \cite{MV}, our proof gives an explicit upper bound on the error term (cf. (\ref{Fn})). This improvement is related to a \emph{fixed point} argument needed for the proof of an inversion procedure, see Section \ref{2} for more details. For the proof of this result, one requires the parameter $\beta>0$ in the Gardner equation (\ref{Ga}) small enough. However, since the formal limit $\beta \to 0$ in (\ref{Ga}) is the KdV equation, the Gardner transform (\ref{Miura}) linking both equations degenerates to the identity and thus does not improve the regularity of the inverse. However, by taking $\al>0$ small, depending on $\beta$ small, we are able to obtain a still satisfactory bound on the stability (\ref{Fn}). 

\medskip

\noindent
2. We do not believe that (\ref{ASn}) holds in the whole real line $\{x\in \R\}$, e.g. based in the Martel-Merle \cite{MMnon} result. Indeed, they have constructed a solution the KdV equation composed of a big soliton plus an infinite train of small solitons, still satisfying the stability property. This implies that there is no strong convergence in $H^1(\R)$ in the general case.

\medskip

Finally, our last result corresponds to the global $L^2$-stability and asymptotic stability of the \emph{$N$-soliton} solution of KdV. 
It turns out that this result is just a direct corollary of  Theorem \ref{T1} and the uniform continuity of the KdV flow for $L^2$-data, as it was pointed out in \cite{MMT}, Corollary 1. We include the proof at the end of Section \ref{3}, for the sake of completeness.

\medskip

\begin{cor}[$L^2$-stability and asymptotic stability of the $N$-soliton of KdV]\label{T3}~

Let  $\delta>0$, $N\geq 2$, $0<c_1^0<\ldots<c_N^0$ and $x_1^0,\ldots,x_N^0\in \R$. There exists $\al_0>0$ such that if $0<\al<\al_0$, then the following holds. Let $u(t)$ be a solution of (\ref{KdV}) such that 
$$
\| u(0)- U^{(N)}(\cdot ; c_j^0,- x_j^0) \|_{L^2(\R)}\leq \al,
$$
with $U^N$ the $N$-soliton profile described in (\ref{NS}). Then there exist $x_j(t)$, $j=1,\ldots, N$, such that
\be\label{stabilitycor}
\sup_{t\in \R} \big\| u(t) - U^{(N)}( \cdot ;c_j^0,-x_j(t)) \big\|_{L^2(\R)} \leq \delta.
\ee
Moreover, there exist $c_j^{+\infty}>0$  such that
\be\label{asympstabcor}
\lim_{t\to +\infty} \big\| u(t) - U^{(N)}(\cdot ; c_j^{+\infty},-x_j(t)) \big\|_{L^2(x> \frac {c_1^0}{10} t )} =0, 
\ee
and $x_j(t)$ are $C^1$ for all $|t|$ large enough, with $x_j'(t)\to c_j^{+\infty}\sim c_j^0$ as $t\to +\infty$. A similar result holds as $t\to -\infty$, with the obvious modifications.
\end{cor}

\medskip

\noindent
{\bf Remark.}  Let us emphasize that the proof of this result requires the \emph{existence} and the \emph{explicit form} of the multi-soliton solution of the KdV equation, and therefore the integrable character of the equation. In particular, we do not believe that a similar result is valid for a completely general, non-integrable gKdV equation, unless one considers some perturbative regimes (cf. \cite{MMcol1,MMcol3} for some global $H^1$-stability results in the non-integrable setting.)



\medskip

\subsection*{Idea of the proofs}
Let us explain the main steps of the proofs. We follow the approach introduced in \cite{MV}; however, in this opportunity, in order to consider the case of several solitons, we introduce some new ingredients:

\medskip

\noindent
1. \emph{The Gardner transform}. First of all, given any $\beta>0$ and $v(t)\in H^1(\R)$, solution of the Gardner equation (\ref{Ga}), the Gardner transform \cite{Ga}
\be\label{Miura}
u(t) =  M_\beta [v](t) :=  [v - \frac 32 \sqrt{2\beta} v_x -\frac 32 \beta v^2](t),
\ee
is an $L^2$-solution of KdV (in the integral sense).\footnote{See Section \ref{5} for additional information about this transform.} Compared with the original Miura transform (\ref{Mi}), it has an additional \emph{linear} term which simplifies the proofs. In particular, a direct computation (see Appendix \ref{A}) shows that for the Gardner soliton solution (\ref{SolG}), one has
\bea\label{QcbQc}
M_\beta [Q_{c,\beta}](t) & = & \big[Q_{c,\beta} - \frac 32 \sqrt{2\beta} Q_{c,\beta}'  -\frac 32 \beta Q_{c,\beta}^2\big] (x-ct) \nonu \\
& = & Q_c(x-ct - \delta ), 
\eea
with $\delta=\delta(c,\beta)>0$ provided $\beta>0$, and $Q_c$ the KdV   soliton solution (\ref{Sol}). In other words, the Gardner transform (\ref{Miura}) sends the Gardner soliton towards a slightly translated KdV soliton. 

\medskip

\noindent 
2. \emph{Lifting}. Given an initial data $u_0$ satisfying (\ref{In}), with $\al>0$ small, we solve  the Ricatti equation $u_0 = M_\beta[v_0]$ in $H^1(\R)$. In addition, we prove that the function $v_0$ is actually close in $H^1(\R)$ to the sum of $N$-solitons of the Gardner equation. However, for the proof of this result, we do not follow the Merle-Vega approach, which is mainly based in a minimization procedure. Instead, we solve the Ricatti equation by using a \emph{fixed point} argument in a neighborhood of $R_0$. It turns out that in order to do this, we need to assume that $\beta$, the free parameter of the Gardner equation, is small enough, and therefore we require $\al$ smaller, depending on $\beta$. In any case, and as a by-product, we obtain explicit bounds on the distance of the solution $v_0$ and the Gardner multi-soliton solution, that one can see in Theorem \ref{T1}. This is done in Section \ref{2}.

\medskip

\noindent 
3. \emph{Conclusion.} Finally, we invoke the $H^1$-stability theory developed by Martel-Merle-Tsai and Martel-Merle  \cite{MMT,MMan}, in the particular case of the Gardner equation. The final conclusion follows directly after a new application of the Gardner transform (\ref{Miura}). This is done in Section \ref{3}. Finally, the global character of the stability and asymptotic stability properties follow after a simple continuity argument applied to the $N$-soliton solution of the KdV equation. This is done at the end of Section \ref{3}.

\medskip

We recall that the proof of Theorem \ref{T1} does not use the \emph{full integrable character} of (\ref{KdV}) and (\ref{Ga}), but only the Gardner transform linking both equations. However, for the proof of Corollary \ref{T3}, we need to work with the $N$-soliton solution. In addition, we simplify and improve the proof of \cite{MV}, since the lifting procedure is easier to prove in the case of localized solutions, and we give an explicit bound in the stability result. It is expected that this method may be applied to others models, see Section \ref{5} for more details.

\bigskip

\section{Lifting}\label{2}

\medskip

Let $u_0\in L^2(\R)$ satisfying (\ref{In}). Let us denote by $z_0 := u_0 -R_0,$ such that $\|z_0\|_{L^2(\R)} \leq \al.$  In this section, our objective is to solve the nonlinear Ricatti equation
\be\label{Ric}
M_\beta [v_0] =  u_0 = R_0 + z_0, 
\ee
with $M_\beta $ the Gardner transform given by (\ref{Miura}). We will do that provided $\al$ is small enough. In other words, we want to solve the Gardner transform in a neighborhood of the multi-soliton solution $R_0$. This is the purpose of the following

\begin{prop}[Local invertibility around $R_0$]\label{P0}~

There exists $\beta_0>0$ such that, for all $0<\beta<\beta_0$, the following holds. There exist $K_0,L_0,\ga_0,\al_0>0$ such that for all $0<\al<\al_0$, $L>L_0$, and $\|z_0\|_{L^2(\R)} \leq \al $, there exists a solution $v_0 \in H^1(\R)$ of (\ref{Ric}), such that 
\be\label{v0}
\big\| v_0 - \sum_{j=1}^N Q_{c_j^0, \beta} (\cdot - x_j^0  - \delta_j) \big\|_{H^1(\R)} \leq  K_0(\frac {\al}{\sqrt{\beta}} + e^{-\ga_0 L}),
\ee
with 
\be\label{dj}
\delta_j = \delta_j(c_j^0):=  (c_j^0)^{-1/2} \cosh^{-1} (\frac 1{\rho_j}),\footnote{We take the positive inverse.} \quad \rho_j := (1 -\frac 92 \beta c_j^0 )^{1/2},\quad  j=1,\ldots, N,
\ee
and $Q_{c,\beta}$ being the soliton solution of the Gardner equation (\ref{Ga}).
\end{prop}

\begin{proof}

\noindent
1. First of all, in what follows we assume $\beta>0$ small in such a way that $\beta <\frac{2}{9c_N^0}$ and $Q_{c_j^0,\beta}$ is well defined for all $j=1,\ldots,N$. Let us consider
$$
S_0(x) :=\sum_{j=1}^N Q_{c_j^0, \beta} (x - x_j^0  - \delta_j),
$$
with $\delta_j$ defined in (\ref{dj}). Let us recall that 
$$
M_\beta[Q_{c_j^0, \beta} (x - x_j^0  - \delta_j)] = Q_{c_j^0}(x-x_j^0),
$$
(cf. Appendix \ref{A}). A Taylor expansion shows that $\delta_j =O(\beta)$, independent of $c_j^0$, as $\beta$ approaches zero. Therefore, in what follows we may suppose that
\be\label{Ma1}
x_j^0 +\delta_j \geq x_{j-1}^0 +\delta_{j-1} + \frac 9{10}L, \quad j=2,\ldots ,N,
\ee
by taking $\beta$ small enough.

\medskip

\noindent
2. It is clear that $S_0\in H^1(\R)$ with $\|S_0\|_{H^1(\R)} \leq K$, independent of $\beta$. Moreover, a direct computation, using (\ref{QcbQc}) and (\ref{Ma1}), shows that
\bea
M_\beta [S_0](t) & = & \sum_{j=1}^N M_\beta [ Q_{c_j^0, \beta} (\cdot - x_j^0  - \delta_j)]  \nonu \\
& & \qquad \qquad - \frac 32 \beta \sum_{i\neq j} Q_{c_i^0, \beta} (\cdot - x_i^0  - \delta_i)Q_{c_j^0, \beta} (\cdot - x_j^0  - \delta_j) \nonu\\
& =& \sum_{j=1}^N Q_{c_j^0}(\cdot -x_j^0)  - \frac 32 \beta \sum_{i\neq j} Q_{c_i^0, \beta} (\cdot - x_i^0  - \delta_i)Q_{c_j^0, \beta} (\cdot - x_j^0  - \delta_j) \nonu \\
& =&  \sum_{j=1}^N Q_{c_j^0}(\cdot -x_j^0)  +O_{L^2(\R)} (\beta e^{-\ga_0 L}) \nonu \\
& = & R_0 +O_{L^2(\R)} (\beta e^{-\ga_0 L}),\label{MR0}
\eea
for some $\ga_0>0$, independent of $\beta$ small. 

\medskip

\noindent
3. Now we look for a solution $v_0\in H^1(\R)$ of (\ref{Ric}), of the form  $v_0 = S_0 + w_0,$ and $w_0$ small in $H^1(\R)$.  In other words, $w_0$ has to solve the nonlinear equation
\be\label{NLP}
\mathcal L[w_0]  =  (R_0- M_\beta [S_0]  ) + z_0 +\frac 32 \beta w_0^2, 
\ee
with 
\be\label{NLP1}
\mathcal L[w_0] :=  - \frac 32 \sqrt{2\beta} w_{0,x} +(1-3\beta S_0)w_0. 
\ee
We may think $\mathcal L$ as a unbounded operator in $L^2(\R)$, with dense domain $H^1(\R)$. From standard energy estimates, one has that for $\beta>0$ small enough, any solution $w_0\in H^1(\R)$ of the linear problem
\be\label{LF}
\mathcal L[w_0] = f, \quad f\in L^2(\R),
\ee
must satisfy
$$
\|(w_0)_x\|_{L^2(\R)} \leq \frac K{\sqrt{\beta}}( \|w_0\|_{L^2(\R)} + \|f\|_{L^2(\R)}),
$$
with $K>0$ independent of $\beta$. On the other hand, to obtain a-priori $L^2$-bounds, note that from  the Young inequality and Plancherel,\footnote{Here $\hat{\cdot}$ denotes the Fourier transform.} 
$$
\| \hat{S}_0 \star \hat{w}_0 \|_{L^2(\R)} \leq \| \hat{S}_0\|_{L^1(\R)} \| \hat{w}_0 \|_{L^2(\R)}.
$$
Since $S_0 $ is in the Schwartz class, one has $\hat{S}_0 \in L^1(\R)$, with uniform bounds. 
By taking $\beta>0$ small and the Fourier transform in (\ref{LF}), one has
$$
(- \frac 32 i\sqrt{2\beta} \xi  +1) \hat{w}_0(\xi) = \hat f (\xi)+  O_{L^2(\R)}(\beta \hat{w}_0).
$$ 
Therefore, using Plancherel,
$$
\|w_0\|_{L^2(\R)} \leq K\|f\|_{L^2(\R)}.
$$
In concluding, one has, for some fixed constant $K_0>0$,
\be\label{solw}
\|w_0\|_{H^1(\R)} \leq \frac {K_0}{\sqrt{\beta}} \|f\|_{L^2(\R)},
\ee
for any $w_0\in H^1(\R)$ solution of (\ref{LF}). In order to prove the existence and uniqueness of a solution of (\ref{LF}), we use a fixed point approach, in the spirit of \cite{W,KM}. Let us introduce the ball
$$
\mathcal B_0 := \Big\{ w_0 \in H^1(\R) \; \Big| \;  \|w_0\|_{H^1(\R)} \leq \frac {K_0}{\sqrt{\beta}}\|f\|_{L^2(\R)} \Big\},
$$
and the complex operator in the Fourier space,
$$
T_0[g](\xi) := \frac{3\beta \hat{S}_0 \star g (\xi) + \hat f(\xi)}{1+\frac 32 i\sqrt{2\beta} \xi }.
$$
It is clear that problem (\ref{LF}) can be written in Fourier variables as the fixed point problem
$$
g = T_0[g], \quad g:= \hat{w}_0.
$$
By simple inspection one can see that $T_0$ is a contraction on $\mathcal B_0$. Indeed, note that for $w_0\in \mathcal B_0$, $g:=\hat{w}_0$,
$$
\|T_0[g] \|_{L^2(\R)} \leq K (\beta \|g\|_{L^2(\R)} + \|f\|_{L^2(\R)}) \leq \frac{K_0}{2\sqrt{\beta}} \|f\|_{L^2(\R)},  
$$
and
$$
\| \xi  T_0[g] \|_{L^2(\R)} \leq K (\beta \| \xi  g\|_{L^2(\R)} +  \frac{1}{\sqrt{\beta}}\|f\|_{L^2(\R)}) \leq \frac{K_0}{2\sqrt{\beta}} \|f\|_{L^2(\R)},
$$
by taking $K_0$ larger. The contraction part works easier. The fixed point theorem gives the existence and uniqueness result. 

\medskip

In what follows, let us denote by $T:= \mathcal L^{-1} : L^2(\R)Ê\to H^1(\R) $ the resolvent operator constructed in step 3.

\medskip

\noindent
4. Finally, from (\ref{NLP}), we want to solve the nonlinear problem
\be\label{FP}
w_0 = T[w_0] = \mathcal L^{-1} \big[  (R_0- M_\beta [S_0]  ) + z_0 +\frac 32 \beta w_0^2 \big]. 
\ee
In order to use, once again, a fixed point argument, let us introduce the ball
$$
\mathcal B := \Big\{ w_0 \in H^1(\R) \; \Big| \;  \|w_0\|_{H^1(\R)} \leq 2K_0(\frac{\al}{\sqrt{\beta}} +e^{-\ga_0 L}) \Big\},
$$
with $K_0>0$ the constant from (\ref{solw}), and $\ga_0>0$ given in (\ref{MR0}). Let $w_0\in \mathcal B$. Note that, from (\ref{FP}), (\ref{MR0}) and (\ref{solw}) 
\bee
\|T[w_0]\|_{H^1(\R)} & \leq & \frac {K_0}{\sqrt{\beta}} [ \| R_0- M_\beta [S_0]  \|_{L^2(\R)} +\al + \beta \|w_0^2\|_{L^2(\R)} ]  \\
& \leq & \frac {K_0}{\sqrt{\beta}} [ K\beta e^{-\ga_0 L} + \al +  4 K_0^2 \beta (\frac{\al}{\sqrt{\beta}} +e^{-\ga_0 L} )^2  ] \\
& \leq &  K_0 ( K\sqrt{\beta} + KK_0 \beta e^{-\ga_0 L} + KK_0\al\sqrt{\beta})e^{-\ga_0 L} \\
& & \qquad + K_0\frac {\al}{\sqrt{\beta}} (1+KK_0\al).
\eee
By taking $\beta_0$ small, and then $\al_0$ smaller if necessary, we can ensure that the above conclusions still hold and therefore
$$
\|T[w_0]\|_{H^1(\R)}   \leq   \frac 32K_0 (\frac{\al}{\sqrt{\beta}} + e^{-\ga_0 L} ). 
$$
This proves that $T(\mathcal B)\subseteq \mathcal B$. In the same way, one can prove that $T$ is a contraction. Indeed, we have for $w_1,w_2\in \mathcal B$,
\bee
\|T[w_1] -T[w_2] \|_{H^1(\R)} & \leq & K_0\beta  \| \mathcal L^{-1} [ w_1^2 -w_2^2] \|_{H^1(\R)}  \\
& \leq &  K K_0 (\frac{\al}{\sqrt{\beta}} +e^{-\ga_0 L}) \beta  \|w_1 -w_2\|_{H^1(\R)} \\
& <& \frac 12 \|w_1 -w_2\|_{H^1(\R)}, 
\eee
provided $ \beta_0$ is small enough. Therefore, $T$ is a contraction mapping from $\mathcal B$ into itself, and there exists a unique fixed point for $T$. The proof is now complete. 
\end{proof}
 


\bigskip

\section{Proof of the Main Theorems}\label{3}

\medskip

In this section we prove Theorem \ref{T1} and Corollary \ref{T3}.

\medskip

\subsection{Proof of Theorem \ref{T1}}~

\medskip

1. Let us assume the hypotheses mentioned in the statement of Theorem \ref{T1}, in particular (\ref{In}). From Proposition \ref{P0}, by taking $\al_0$ smaller if necessary, there exist $\beta>0$ small, and  $v_0\in H^1(\R)$, solution of the Ricatti equation (\ref{Ric}), which satisfies (\ref{v0}).

\medskip

Next, we recall the following $H^1$-stability result valid for the Gardner equation.

\begin{prop}[$H^1$-stability for Gardner solitons, \cite{MMT,MMan}]\label{P1}~

Let $0<c_1^0<c_2^0<\ldots <c_N^0<\frac{2}{9\beta}$ be such that 
\be\label{WC}
\partial_c\int_\R Q_{c,\beta}^2 \Big|_{c=c_j}>0, \quad \hbox{for all } j=1,\ldots, N. \qquad (\hbox{Weinstein's criterium.})
\ee
There exists $\tilde \al_0, \tilde A_0, \tilde L_0, \tilde\ga>0$ such that the following is true.  Let $v_0 \in H^1(\R)$, and assume that there exists $\tilde L> \tilde L_0$, $\tilde \al \in (0,\tilde \al_0)$ and $\tilde x_1^0< \tilde x_2^0<\ldots < \tilde x_N^0$, such that 
\bea\label{Inv}
& & \|  v_0 -\sum_{j=1}^N Q_{c_j^0,\beta} (\cdot - \tilde x_j^0 ) \|_{H^1(\R)} \leq \tilde \al, \\
&  &  \quad \tilde x_{j}^0> \tilde x_{j-1}^0 + \tilde L, \quad j=2, \ldots, N.
\eea
Then there exists $\tilde x_1(t), \ldots \tilde x_N(t)$ such that the solution $v(t)$ of the Cauchy problem associated to (\ref{Ga}), with initial data $v_0$, satisfies 
$$
v(t) = S(t) + w(t), \quad S(t) :=\sum_{j=1}^N Q_{c_j^0,\beta}(\cdot - \tilde x_j(t) ),
$$
and
\be\label{Fnv}
\sup_{t\geq 0} \Big\{ \| w(t)\|_{H^1(\R)} + \sum_{j=1}^N  |\tilde x'_j(t)-c_j|   \Big\} \leq \tilde A_0 (\tilde \al+ e^{-\tilde \ga \tilde L}).
\ee
\end{prop}

\begin{proof}
Although this proof is not present in the literature, it is a direct consequence of \cite{MMT} (see also Section 5 in \cite{MMan}.) For the proof of (\ref{WC}), note that from (\ref{SolG})
\be\label{WC1}
\partial_c \int_\R Q_{c,\beta}^2 =  \frac 32c^{1/2} \int_\R Q^2 + O(\beta)>0, 
\ee
for $\beta$ small. See also \cite{Ale} for the explicit computation.
\end{proof}

2. Since $v_0$ satisfies (\ref{v0}), by taking $\al_0>0$ smaller and $L_0$ larger if necessary, we can apply the above Proposition with 
\bea
& & \tilde \al := K_0 (\frac{\al}{\sqrt{\beta}} + e^{-\ga_0 L}), \quad \tilde L := \frac 9{10}L, \label{tal}\\
& &  \tilde x_j^0 := x_j^0 +\delta_j, \; j=2,\ldots, N.\nonu
\eea
Therefore, there exist $\tilde A_0>0$, parameters $\tilde x_j(t)\in \R$ and a solution $v(t)$ of (\ref{Ga}), defined for all $t\geq 0$, and satisfying 
\be\label{KA}
\sup_{t\geq 0}\big\| v(t) - \sum_{j=1}^N Q_{c_j^0,\beta}(\cdot - \tilde x_j(t) ) \big\|_{H^1(\R)}\leq \tilde A_0 (\al + e^{-\ga L}),
\ee
for some $\ga>0$ and $\tilde A_0=\tilde A_0(\beta)$ (note that $\tilde L$ and $L$ are of similar size).

\medskip

Now we are ready to prove the first part of Theorem \ref{T1}.

\medskip

3. {\bf $L^2$-stability.} The final steps of the stability proof are similar to those followed in \cite{MV}: Let us define
$$
\bar u(t) := M_\beta [v](t).
$$
with $M_\beta$ given in (\ref{Miura}). Note that 
\ben
\item The initial datum satisfy
$$
\bar u(0) =M_\beta [v](0) = M_\beta[v_0] =u_0 = R_0 +z_0.
$$
\item $\bar u (t)$ is an $L^2$-solution of the KdV equation (\ref{KdV}).
\item From the definition of $M_\beta[v](t)$ and (\ref{Fnv}), one has
\bee
\bar u (t) & = &  M_\beta[S(t) + w(t)] \\
& = & M_\beta [S](t) +  M_\beta[w](t) -3\beta S(t)w(t).
\eee
\een
Let us consider this last term. From (\ref{KA}), one has 
$$
\| M_\beta[w](t) -3\beta S(t)w(t)\|_{L^2(\R)} \leq \tilde A_0 (\al + e^{-\ga L}),
$$
and, similarly to (\ref{MR0}),
\bee
M_\beta [S](t) & = & \sum_{j=1}^N M_\beta[ Q_{c_j^0,\beta}(\cdot - \tilde x_j(t) )]  \\
& & \qquad -\frac 32\beta \sum_{i\neq j}  Q_{c_i^0,\beta}(\cdot - \tilde x_i(t) ) Q_{c_j^0,\beta}(\cdot - \tilde x_j(t) ) \\
& =& \sum_{j=1}^N Q_{c_j^0} (\cdot - \tilde x_j(t) - \delta_j)  + O_{L^2(\R)} (\beta e^{-\ga \tilde L}) \\
& = :& \sum_{j=1}^N Q_{c_j^0} (\cdot - x_j(t))  + O_{L^2(\R)} (\beta e^{-\ga  L}),
\eee
with $x_j(t) := \tilde x_j(t) + \delta_j$. 
Therefore, the final conclusion  follows from the uniqueness of $u(t)$, solution of (\ref{KdV}) with initial data $u_0$ \cite{B1}. 

\medskip

4. {\bf Asymptotic stability in $L^2(\R)$.}  Finally, in this paragraph we prove that solitons are asymptotically stable in $L^2(\R)$, in the sense of Martel-Merle, namely estimate (\ref{ASn}). For this purpose, we recall the following result proved in \cite{MMan} (see also Remark 3 in that paper).

\begin{prop}[Asymptotic stability in $H^1(\R)$,  \cite{MMan}]\label{ASMM}~

Suppose that (\ref{Fnv}) holds. Then, there exist $c_j =  c_j(t) \in (0, \frac{2}{9\beta}) $, and $\rho_j(t) \in \R$, $j=1,\ldots , N$, such that 
\be\label{cj}
| c_j(t) -c_j^0 | + |c_j(t) - \rho_j'(t)| \leq K \tilde A_0 (\tilde \al + e^{-\tilde \ga \tilde L}),
\ee
and
\be\label{AS1}
\lim_{t\to +\infty} \big\| v(t)  -  \sum_{j=1}^N Q_{c_j,\beta} (\cdot -\rho_j(t))  \big\|_{H^1(x>\frac {c_1^0}{10} t )} =0.
\ee
Moreover, there exist $c_j^+ \in (0, \frac{2}{9\beta})$ such that $c_j(t)\to c_j^+$  as $t\to +\infty$, and $j=1,\ldots, N$.
\end{prop}

Let us recall that the last conclusion above is a consequence of the fact that the integral $\int_{\R } Q_{c,\beta}^2 $ varies with $c$ (see (\ref{WC1})), which is a sufficient condition to obtain the convergence of the scaling parameters.

\medskip

From the above result we can define
$$
\tilde w(t) := v(t)  -  \sum_{j=1}^N Q_{c_j,\beta} (\cdot -\rho_j(t)),
$$
such that 
\be\label{Limi}
\lim_{t\to +\infty}\|\tilde w(t)\|_{H^1(x\geq \frac 1{10} c_1^0 t)} =0.
\ee
Using the Gardner transform, we know that for $\delta_j(t) := \delta_j (c_j(t))$ (cf. Proposition \ref{P0}),
\bee
u(t) &  = & M_\beta [v](t) \\
& = &  M_\beta[\sum_{j=1}^N Q_{c_j,\beta} (\cdot -\rho_j(t))] + M_\beta[\tilde w](t) -3\beta \tilde w(t) \sum_{j=1}^N Q_{c_j,\beta} (\cdot -\rho_j(t)) \\
& =& \sum_{j=1}^N Q_{c_j} (\cdot -\rho_j(t) +\delta_j(t))  -3\beta \sum_{i\neq j}^N Q_{c_i,\beta} (\cdot -\rho_i(t))Q_{c_j,\beta} (\cdot -\rho_j(t)) \\
& & \qquad + M_\beta[\tilde w](t) -3\beta \tilde w(t) \sum_{j=1}^N Q_{c_j,\beta} (\cdot -\rho_j(t)). 
\eee
Now, it is clear from (\ref{cj}) that 
$$
\lim_{t\to +\infty}\big\| \sum_{i\neq j}^N Q_{c_i,\beta} (\cdot -\rho_i(t))Q_{c_j,\beta} (\cdot -\rho_j(t)) \big\|_{L^2(x>\frac {c_1^0}{10}  t)} =0.
$$
On the other hand, from (\ref{Limi}) one has
$$
\lim_{t\to +\infty}\big\| M_\beta[\tilde w](t) -3\beta \tilde w(t) \sum_{j=1}^N Q_{c_j,\beta} (\cdot -\rho_j(t)) \big\|_{L^2(x>\frac {c_1^0}{10}  t)} =0.
$$
Finally, by redefining $x_j(t) := \rho_j(t) -\delta_j(t)$, using (\ref{cj}) and an argument similar to Step 1 in the proof of Proposition \ref{P0}, we obtain the final conclusion. The proof is complete.

\medskip

\noindent
{\bf Remark.} It is important to stress that the invertibility property above mentioned in Proposition \ref{P0}  depends on $\beta$ small, and it should be present in the main result, namely Theorem \ref{T1}. In fact,  we have chosen $\al_0$ depending on $\beta$ such that $\tilde \al$ in (\ref{tal}) is small enough to apply the stability result for the Gardner equation. Therefore, in (\ref{Fn}) the dependence in $\beta$ is hidden under the constant $A_0$.

\bigskip

\subsection{The case of negative times} One may concern whether the preceding result, valid for positive times, can be extended as in \cite{MS}, for negative times, or even better, for all time. We have a first, positive answer for this question. Indeed, by using a continuity argument inside the \emph{interaction region} and the explicit multi-soliton solution of the KdV equation, one has the following

\medskip

\begin{prop}[$L^2$-stability for negative times]\label{T2}~

Let $\delta>0$ fixed. Under the hypotheses of Theorem \ref{T1}, by taking $\al_0$ smaller and $L_0$ larger if necessary, there exist $\tilde T\geq 0$ and $x_j(t)\in \R$, $j=1,\ldots, N$, defined for all $|t|\geq \tilde T$, and such that 
\be\label{FnG}
\sup_{t\leq -\tilde T}\big\| u(t) - \sum_{j=1}^N Q_{c_j^0}(\cdot -x_j(t))  \big\|_{L^2(\R)}\leq \delta.
\ee
Moreover, as above, the asymptotic stability result (\ref{ASn}) can be extended as $t\to -\infty$, with the obvious modifications.
\end{prop}

\begin{proof}[Proof of Proposition \ref{T2}]~

We use the notation introduced in \cite{MMcol1}. Let  $\delta>0$ fixed, and let 
$$
T := T(x_1^0,\ldots,x_N^0; c_1^0,\ldots, c_N^0) <0
$$
be the first \emph{interaction time} among the solitons. In particular, for $t \leq T$, solitons are well ordered and separated (in terms of their mutual distance  $L$), but with the inverse order compared with case of positive times. Note that this definition depends only on the set $(c_j^0, x_j^0)_{j=1,\ldots,N}$. By taking  $L_0$ larger if necessary, one has from (\ref{In}) and the explicit form of $U^{(N)}$,
$$
\|u_0 - U^{(N)}(\cdot ; c_j^0,-x_j^0) \|_{L^2(\R)} \leq 2\al.
$$
Let $\tilde u(t,x) := U^{(N)}(x ; c_j^0,-(x_j^0+c_j^0 t))$ be the $N$-soliton solution associated to the initial datum $U^{(N)}(x ; c_j^0,-x_j^0)$, \cite{MS}. From the uniform, continuous dependence on the initial datum in $L^2(\R)$ of the KdV equation (cf. \cite{B1}), one has that
$$
\| u(t) - \tilde u(t) \|_{L^2(\R)} \leq \delta,
$$ 
for all $t\in [T, 0]$, provided $\al_0$ is chosen small enough. However, from the definition of $T$ and a computation one has that
$$
\big\| \tilde u(T)  - \sum_{j=1}^N Q_{c_j^0} ( \cdot -x_j^+ -c_j^0 T) \big\|_{L^2(\R)} \leq K e^{-\ga L},
$$
for some $\ga>0$, and where $x_j^+ = x_j^+ ((c_i^0),(x_i^0))$ are the shifts induced by the elastic collision \cite{HIROTA}.  Note that by definition of $T$, each soliton is well ordered and separated for $t\leq T$ (in the inverse sense compared with $t\geq 0$), and therefore $x_j^+ + c_j^0 T + L \leq x_{j-1}^+ + c_{j-1}^0 T $, for all $j=2,\ldots, N$. Therefore, by taking $\al_0$ smaller and $L_0$ larger if necessary, we can apply Theorem \ref{T1} backwards in time (just note that $u(-t,-x)$ is also a solution of KdV) to conclude the proof. 

\end{proof}
\medskip

\noindent
{\bf Remark.} We could have used alternatively the $H^1$-local well-posedness theory given in \cite{KPV3} for the Gardner equation, and then the Gardner transform to obtain a similar result as above.

\bigskip

\subsection{Proof of Corollary \ref{T3}}~

We follow the proof of Corollary 1  in \cite{MMT}. The proof is also similar to the proof of Proposition \ref{T2}. First note that the $N$-soliton behaves as the sum of $N$-soliton as the distance among each soliton diverges. Indeed,
\be\label{solprofile}
\lim_{\inf(y_{j+1}-y_j)\to +\infty}\big\| U^{(N)}( \cdot ; c_j^0,-y_j) -\sum_{j=1}^N Q_{c_j^0}(\cdot -y_j) \big\|_{L^2(\R)} =0.
\ee
Let $\delta>0$ be a small fixed number. For $\gamma_0$, $A_0$, $L_0$ and $\alpha_0$ as in the statement of
Theorem \ref{T1}, let $ \alpha_1<\alpha_0$, $L>L_0$ be such that
$A_0 ( \alpha_1 + e^{-\gamma_0 L}  ) < \frac 12\delta$ and
\be\label{defofL}
\big\| U^{(N)}(\cdot; c_j^0,-y_j) -\sum_{j=1}^N Q_{c_j^0}(\cdot -y_j) \big\|_{L^2(\R)} \leq \frac 12 \delta,
\ee
for $y_{j+1}-y_j>L$. We may suppose $A_0\geq1$.

Now, let $\tilde u(t,x) := U^{(N)}(x ; c_j^0,-(x_j^0+c_j^0 t))$ be the $N$-soliton solution of (\ref{KdV}) with initial datum $U^{(N)}(\cdot ; c_j^0,-x_j^0)$. Let  $\tilde T = \tilde T(\al_1,L)>0$ be such that, for all $t\geq \tilde T$,
\be\label{by}
\big\| \tilde u(t)-\sum_{j=1}^N Q_{c_j^0}(\cdot-(x_j^0+c_j^0 t))\big\|_{L^2(\R)} \leq \frac 12\alpha_1,
\ee
and for all $j$,  $x_{j+1}^0 + c_{j+1}^0 \tilde T \ge x_{j}^0 + c_{j}^0 \tilde T + 2L.$

Therefore, by the uniform continuous dependence of the solution  of (\ref{KdV}) with respect to the initial datum in $L^2(\R)$ (see \cite{B1}), there exists $\alpha>0$ such that if  $\|u(0)-\tilde u (0)\|_{L^2(\R)}\le \alpha$, then for all $t\in [0, \tilde T]$, one has $\| u(t)- \tilde u(t)\|_{L^2(\R)}\leq \frac 12 \alpha_1 < \delta $. In particular, by (\ref{by}), 
$$
\big\|u(\tilde T)-\sum_{j=1}^N Q_{c_j^0}(\cdot-(x_j^0+c_j^0 \tilde T))\big\|_{L^2(\R)} \leq \alpha_1<\al_0.
$$
Thus, by Theorem \ref{T1}, there exist $x_j(t)$, such that for all $t\geq \tilde T$,
$$
\big\| u(t)-\sum_{j=1}^N Q_{c_j^0}(\cdot- x_j(t)) \big\|_{L^2(\R)} \le A_0 ( \alpha_1 + e^{-\gamma_0 L} ) < \frac 12\delta.
$$
Moreover, $x_{j+1}(t)>x_j(t)+L$. Together with (\ref{defofL}), this gives the stability result for positive times $t\geq \tilde T$.

Next, taking $\al$ smaller if necessary, we can use a similar argument to that in the proof of Proposition \ref{T2} to extend the stability result backwards in time, until before the first collision time. Finally, we extend the result for all negative times by using again Theorem \ref{T1}.

\medskip

Finally, the asymptotic stability result follows from (\ref{solprofile}) and the second part of Theorem \ref{T1}. The proof is complete.

\bigskip

\section{Final remarks}\label{5}

\medskip

In this last section we would like to stress some points which are of independent interest.

\medskip

\noindent
{\bf 1. On the relationship between the Miura and Gardner transforms}. There is a second way to see the transform (\ref{Miura}), which uses the standard Miura transform $M$, considered in (\ref{Mi}). This way is probably not new in the literature, so we recall it by completeness purposes. 

Indeed, let $v\in H^1(\R)$ be a solution of (\ref{Ga}). Then, the auxiliary function 
$$
\tilde v(t,x):= \frac{1}{3\sqrt{\beta}} -  \sqrt{\beta} \, v \big(t,x+\frac{t}{3\beta } \big) 
$$
solves the mKdV equation (\ref{dmKdV}). Note that $\tilde v$ is a $L^\infty$-function with nonzero limits at infinity. Next, $M[\tilde v]$ is a solution of (\ref{KdV}), with nonzero limit at infinity. Using the fact that (\ref{KdV}) is \emph{Galilean invariant}, one has that 
$$
u(t,x)=\frac{1}{6\beta} + M[v] \big(t, x -\frac{t}{3\beta}\big),
$$ 
is also an $L^2(\R)$-solution of KdV. Finally, one can easily check that the composition of these applications gives (\ref{Miura}). See e.g. \cite{Ale} for further applications of this transform.

\medskip

\noindent
{\bf 2. The KP II model.}\footnote{In this paragraph we follow the notation of \cite{MT}.} It is also interesting to stress that, modulo a constant transformation on the scaling, the KdV soliton $Q_c(x-ct)$, defined in (\ref{Q}), and seen as a two-variable function of $x$ and $y$, is a non-localized solution of the KP II equation
\bea\label{KPII}
& & u_t + (u_{xx} + 3u^2)_x + 3v_{y} =0, \quad \hbox{ in } \quad \R_t \times \R_{x,y}^2,\\
& & u=u(t,x,y), \quad v:= \partial_x^{-1}u_y. \nonu
\eea
In \cite{MT}, Mizumachi and Tzvetkov follow the idea of Merle and Vega and perform a Miura transform to show that $Q_c$ is stable under small perturbations in the space $L^2(\R_x\times \T_y)$, where here $\T_y$ is the one-dimensional torus in the $y$-variable. The pivot equation is now the integrable model called \emph{modified} KP II equation, which is given by 
$$
u_t + (u_{xx} - 2u^3)_x  + 3 v_{y} + 6u_x v=0, \quad \hbox{ in } \quad \R_t \times \R_{x,y}^2,
$$
with $v$ defined in (\ref{KPII}). Note in addition that, after a standard scaling modification, the one-variable kink solution of the mKdV equation (\ref{dmKdV}) is also an admissible solution of this last equation (seen as a two-variable function).

\medskip

We believe that, using the methods developed in this paper, the stability --under periodic transversal perturbations-- of the KdV multi-soliton $U^{(N)}$, seen as a solution of (\ref{KPII}), constant in the $y$-variable, can be handled via a Gardner transform pointing this time to the \emph{integrable}, Gardner generalization of KP II, namely
\bee
& & \tilde u_t + (\tilde u_{xx}  + 3\tilde u^2 - 2 \tilde u^3)_x + 3\tilde v_{y} +6\tilde u_x \tilde v =0,\quad \hbox{ in } \quad \R_t \times \R_{x,y}^2, \\
 & & \tilde u=\tilde u (t,x,y), \quad  \tilde v:= \partial_x^{-1} \tilde u_y,
\eee
for which a scaled version of the Gardner soliton (\ref{SolG}), seen as a constant function in the $y$-variable, is a simple solution. The Gardner-KP II transform is given in this case by the simple formula
$$
\tilde M[\tilde u] :=  \tilde u + \tilde u_x + \tilde v -\tilde u^2.  
$$

\bigskip

\appendix

\section{Proof of (\ref{QcbQc})}\label{A}

In this small paragraph we prove, for the sake of completeness, that the Gardner transform sends solitons of the Gardner equation towards translated solitons of KdV, namely the identity 
(\ref{QcbQc}). Let us recall that $\rho = (1-\frac 92\beta c)^{1/2}$.  Indeed, note that by (\ref{Miura}) and (\ref{SolG}),
\bee
& & M_\beta [Q_{c,\beta}](t,x) = \\
& & =  \big[Q_{c,\beta} - \frac 32 \sqrt{2\beta} Q_{c,\beta}'  -\frac 32 \beta Q_{c,\beta}^2\big] (x-ct) \\
& & = \Big[Ê\frac{3c}{1+ \rho \cosh(\sqrt{c}s)} +   \frac 92 \sqrt{2\beta}c^{3/2} \rho  \frac{ \sinh(\sqrt{c}s)}{(1+ \rho \cosh(\sqrt{c}s))^2 }  \\
& & \qquad - \frac{27 \beta c^2}{2(1+ \rho \cosh(\sqrt{c}s))^2}\Big]\Big|_{s=x-ct} \\
& & = \frac{3c \rho}{(1+ \rho \cosh(\sqrt{c}s))^2} \big[  \rho +  \cosh(\sqrt{c}s) +   \frac 32 \sqrt{2\beta c} \sinh(\sqrt{c}s)   \big] \Big|_{s=x-ct}.
\eee
Now, let us note that for $\beta>0$ one has $\rho<1$ and therefore $\delta := \frac 1{\sqrt{c}} \cosh^{-1}(\frac 1\rho)>0$ is a well defined quantity, provided we take e.g. the positive inverse of $\cosh$. Note that the shift in the KdV soliton is \emph{always} present since $\beta>0$.  Moreover, with this choice one has
\be\label{SH}
\cosh (\sqrt{c} \delta) = \frac 1\rho, \qquad \sinh(\sqrt{c} \delta) = \frac 1\rho \sqrt{1-\rho^2} = \frac 3{2\rho} \sqrt{2\beta c} >0.
\ee
We replace these identities above, to obtain
\bea
& & M_\beta [Q_{c,\beta}](t,x) =  \nonu\\
& & =  \frac{3c }{\cosh^2(\sqrt{c\delta}) \big[1+  \frac{\cosh(\sqrt{c}s)}{\cosh(\sqrt{c\delta})}  \big]^2} \times \nonu\\
& & \qquad \times \big[  1 + \cosh(\sqrt{c}\delta) \cosh(\sqrt{c}s) +    \sinh (\sqrt{c} \delta) \sinh(\sqrt{c}s)   \big]\Big|_{s=x-ct}  \nonu \\
& & = 3c \frac{ (1 + \cosh(\sqrt{c}\delta) \cosh(\sqrt{c}s) +    \sinh (\sqrt{c} \delta) \sinh(\sqrt{c}s) )}{1+ \sinh^2(\sqrt{c\delta}) +  \cosh^2(\sqrt{c}s) + 2\cosh(\sqrt{c\delta})\cosh(\sqrt{c}s) }\Big|_{s=x-ct}.\nonu \\
& &  \label{marca}
\eea
Note that 
\bea\label{marca2}
& & 1 + \cosh(\sqrt{c}\delta) \cosh(\sqrt{c}(x-ct)) +    \sinh (\sqrt{c} \delta) \sinh(\sqrt{c}(x-ct)) \nonu \\
& & \qquad =1 +\cosh(\sqrt{c}(x-ct +\delta ))>0,
\eea
and
\bee
& & (1 + \cosh(\sqrt{c}\delta) \cosh(\sqrt{c}(x-ct)) +    \sinh (\sqrt{c} \delta) \sinh(\sqrt{c}(x-ct)) ) \times \\
& & \qquad \times  (1 + \cosh(\sqrt{c}\delta) \cosh(\sqrt{c}(x-ct)) -    \sinh (\sqrt{c} \delta) \sinh(\sqrt{c}(x-ct)) ) =\\
& & =  (1 + \cosh(\sqrt{c}\delta) \cosh(\sqrt{c}(x-ct)) )^2 - \sinh^2 (\sqrt{c} \delta) \sinh^2(\sqrt{c}(x-ct)) ) \\
& & = 1 + 2\cosh(\sqrt{c}\delta) \cosh(\sqrt{c}(x-ct)) +  \cosh^2(\sqrt{c}\delta) \cosh^2(\sqrt{c}(x-ct)) \\
& & \qquad - \sinh^2 (\sqrt{c} \delta) \sinh^2(\sqrt{c}(x-ct)) ) \\
& & = 1 + 2\cosh(\sqrt{c}\delta) \cosh(\sqrt{c}(x-ct)) + \cosh^2(\sqrt{c}(x-ct)) +  \sinh^2(\sqrt{c}\delta),  
\eee
which is the denominator in (\ref{marca}). Therefore, from (\ref{marca2}) we can simplify this term to obtain
\bee
& & M_\beta [Q_{c,\beta}](t,x) = \\
& & =  \frac{ 3c }{1 + \cosh(\sqrt{c}\delta) \cosh(\sqrt{c}(x-ct)) -    \sinh (\sqrt{c} \delta) \sinh(\sqrt{c}(x-ct))} \\
& & = \frac{ 3c }{1 + \cosh(\sqrt{c}(x-ct -\delta))} \\
& & = Q_c(x-ct-\delta),
\eee
as desired (cf. (\ref{Sol})-(\ref{Q})). An  a-posteriori analysis shows that the final result is independent of the sign chosen for $\delta$, provided $\sinh (\sqrt{c} \delta)$ is chosen negative in (\ref{SH}).

\medskip

\noindent
{\bf Acknowdlegments.} We would like to thank Yvan Martel and Frank Merle for several and useful comments on a first version of this paper, and Y. Martel for the link with the KP II equation. We also thank the referee for some important remarks which helped to improve the quality of this paper.

\bigskip

\end{document}